\documentclass{amsart}%
\usepackage{amsfonts}
\usepackage{amsmath}
\usepackage{amssymb}
\usepackage{graphicx}%
\providecommand{\U}[1]{\protect\rule{.1in}{.1in}}
\newtheorem{theorem}{Theorem}
\theoremstyle{plain}
\newtheorem{acknowledgement}{Acknowledgement}

\newtheorem{corollary}{Corollary}

\newtheorem{lemma}{Lemma}

\newtheorem{proposition}{Proposition}
\newtheorem{remark}{Remark}

\numberwithin{equation}{section}
\begin{document}
\title[Maximal Covariance]{Maximal Covariance Group of Wigner Transforms and Pseudo-Differential Operators}
\author{Nuno Costa Dias}
\address[A. One and A. Three]{Departamento de Matem\'{a}tica. Universidade
Lus\'{o}fona. Av. Campo Grande, 376, 1749-024 Lisboa}
\email[A. One]{ncdias@meo.pt}
\urladdr{}
\author{Maurice A. de Gosson}
\address[A. Two]{University of Vienna, Faculty of Mathematics, NuHAG\\
Nordbergstrasse 15, A-1090 Vienna}
\email[A.~Two]{maurice.de.gosson@univie.ac.at}
\urladdr{}
\author{Jo\~{a}o Nuno Prata}
\email[A.~Three]{joao.prata@mail.telepac.pt }
\urladdr{}
\thanks{}
\thanks{}
\thanks{This paper is in final form and no version of it will be submitted for
publication elsewhere.}
\date{}
\dedicatory{ }
\begin{abstract}
We show that the linear symplectic and anti-symplectic transformations form
the maximal covariance group for both the Wigner transform and Weyl operators.
The proof is based on a new result from symplectic geometry which
characterizes symplectic and anti-symplectic matrices, and which allows us, in
addition, to refine a classical result on the preservation of symplectic
capacities of ellipsoids.

\end{abstract}
\maketitle


\section*{Introduction}

It is well known \cite{Folland,Birk,Birkbis,Wong} that the Wigner transform
\begin{equation}
W\psi(x,p)=\left(  \tfrac{1}{2\pi\hbar}\right)  ^{n}\int_{\mathbb{R}^{n}%
}e^{-\frac{i}{\hbar}p\cdot y}\psi(x+\tfrac{1}{2}y)\overline{\psi}(x-\tfrac
{1}{2}y)dy \label{wigdef}%
\end{equation}
of a function $\psi\in\mathcal{S}(\mathbb{R}^{n})$ has the following
\textit{symplectic covariance }property: let $S$ be a linear symplectic
automorphism of $\mathbb{R}^{2n}$ (equipped with its standard symplectic
structure) and $\widehat{S}$ one of the two metaplectic operators covering
$S$; then
\begin{equation}
W\psi\circ S=W(\widehat{S}^{-1}\psi). \label{wigco2}%
\end{equation}
It has been a long-standing question whether this property can be generalized
in some way to arbitrary non-linear symplectomorphisms (the question actually
harks back to the early days of quantum mechanics, following a question of
Dirac \cite{Dirac,dragt}). In a recent paper \cite{JPDOA} one of us has shown
that one cannot expect to find an operator $\widehat{F}$ (unitary, or not)
such that $W\psi\circ F^{-1}=W(\widehat{F}\psi)$ for all $\psi\in
\mathcal{S}^{\prime}(\mathbb{R}^{n})$ when $F\in\operatorname*{Symp}(n)$ (the
group of all symplectomorphisms of the standard symplectic space) unless $F$
is linear (or affine). In this paper we show that one cannot expect to have
covariance for arbitrary linear automorphisms of $\mathbb{R}^{2n}$. More
specifically: fix $M\in GL(2n,\mathbb{R})$, and suppose that for any $\psi
\in\mathcal{S}^{\prime}(\mathbb{R}^{n})$ there exists $\psi^{\prime}%
\in\mathcal{S}^{\prime}(\mathbb{R}^{n})$ such that
\[
W\psi\circ M=W\psi^{\prime}%
\]
then $M$ is either symplectic, or antisymplectic (\textit{i.e.} $MC$ is
symplectic, where $C$ is the reflection $(x,p)\longmapsto(x,-p)$ ) (Theorem
\ref{Theo1}).

The covariance property (\ref{wigco2}) is intimately related to the following
property of Weyl operators: assume that $\widehat{A}$ is a continuous linear
operator $\mathcal{S}(\mathbb{R}^{n})\longrightarrow\mathcal{S}^{\prime
}(\mathbb{R}^{n})$ with Weyl symbol $a$; writing this correspondence
$\widehat{A}\overset{\mathrm{Weyl}}{\longleftrightarrow}a$ we then have
\begin{equation}
\widehat{S}^{-1}\widehat{A}\widehat{S}\overset{\mathrm{Weyl}}%
{\longleftrightarrow}a\circ S \label{weylco1}%
\end{equation}
for every symplectic automorphism $S$. Properties (\ref{wigco2}) and
(\ref{weylco1}) are in fact easily deduced from one another. One shows
\cite{Stein,Wong} that property (\ref{weylco1}) is really characteristic of
Weyl calculus: it is the only pseudo-differential calculus enjoying this
symplectic covariance property (however, see \cite{transam} for partial
covariance results for Shubin operators). We will see, as a consequence of our
study of the Wigner function, that one cannot extend property (\ref{weylco1})
to non-symplectic automorphisms. More precisely, if $S$ is not a symplectic or
antisymplectic matrix, then there exists no unitary operator $\widehat{S}$
such that (\ref{weylco1}) holds for all $\widehat{A}\overset{\mathrm{Weyl}%
}{\longleftrightarrow}a$. In other words, the group of linear symplectic and
antisymplectic transformations is the maximal covariance group for the
Weyl--Wigner calculus.

It turns out that the methods we use allow us in addition to substantially
improve a result from symplectic topology. Recall that a symplectic capacity
on $\mathbb{R}^{2n}$ is a mapping $c$ associating to every subset
$\Omega\subset\mathbb{R}^{2n}$ a nonnegative number, or $+\infty$, and
satisfying the following properties:

\begin{itemize}
\item \emph{Symplectic invariance: }$c(f(\Omega))=c(\Omega)$\ \textit{if }
$f\in\operatorname*{Symp}(n)$\emph{;}

\item \emph{Monotonicity: }$\Omega\subset\Omega^{\prime}\Longrightarrow
c(\Omega)\leq c(\Omega^{\prime})$ \textit{ }

\item \emph{Conformality: }$c(\lambda\Omega)=\lambda^{2}c(\Omega)$
\ \textit{for every} $\lambda\in\mathbb{R}$

\item \emph{Nontriviality and normalization:}
\[
c(B^{2n}(R))=\pi R^{2}=c(Z_{j}^{2n}(R))
\]
where $B^{2n}(R)$ is the ball $|z|\leq R$ and $Z_{j}^{2n}(R)$ is the cylinder
$x_{j}^{2}+p_{j}^{2}\leq R^{2}$.
\end{itemize}

An important property is that all symplectic capacities agree on ellipsoids in
$\mathbb{R}^{2n}$. Now, a well known result is that if $f\in GL(2n,\mathbb{R}%
)$ preserves the symplectic capacity of \emph{all} ellipsoids in
$\mathbb{R}^{2n}$ then $f$ is either symplectic or antisymplectic (the notion
will be defined below). It turns out that our Lemma \ref{Lemma} which we use
to prove our main results about covariance yields the following sharper result
(Proposition \ref{prop1}): if $f$ preserves the symplectic capacity of all
symplectic balls, then $f$ is either symplectic or antisymplectic (a
symplectic ball is an ellipsoid which is the image of $B^{2n}(R)$ by a linear
symplectic automorphism; see section \ref{seclemma}).

\noindent\textbf{Notation and Terminology.} The standard symplectic form on
$\mathbb{R}^{2n}\equiv T^{\ast}\mathbb{R}^{n}$ is defined by $\sigma
(z,z^{\prime})=p\cdot x^{\prime}-p^{\prime}\cdot x$ if $z=(x,p)$, $z^{\prime
}=(x^{\prime},p^{\prime})$. An automorphism $S$ of $\mathbb{R}^{2n}$ is
symplectic if $\sigma(Sz,Sz^{\prime})=\sigma(z,z^{\prime})$ for all
$z,z^{\prime}\in\mathbb{R}^{2n}$. These automorphisms form a group
$\operatorname*{Sp}(n)$ (the standard symplectic group). The metaplectic group
$\operatorname*{Mp}(n)$ is a group of unitary operators on $L^{2}%
(\mathbb{R}^{n})$ isomorphic to the double cover $\operatorname*{Sp}_{2}(n)$
of the symplectic group. The standard symplectic matrix is $J=%
\begin{pmatrix}
0 & I\\
-I & 0
\end{pmatrix}
$ where $I$ (resp. $0$) is the $n\times n$ identity (resp. zero) matrix. We
have $\sigma(z,z^{\prime})=Jz\cdot z^{\prime}=(z^{\prime})^{T}Jz$ and
$S\in\operatorname*{Sp}(n)$ if and only if $S^{T}JS=J$ (or, equivalently,
$SJS^{T}=J$).

\section{A Result About Symplectic Matrices}

\subsection{Two symplectic diagonalization results}

We denote by $\operatorname*{Sp}^{+}(n)$ the subset of $\operatorname*{Sp}(n)$
consisting of symmetric positive definite symplectic matrices. We recall that
if $G\in\operatorname*{Sp}^{+}(n)$ then $G^{\alpha}\in\operatorname*{Sp}%
^{+}(n)$ for every $\alpha\in\mathbb{R}$. We also recall that the unitary
group $U(n,\mathbb{C})$ is identified with the subgroup
\begin{equation}
U(n)=\operatorname*{Sp}(n)\cap O(2n,\mathbb{R}) \label{un}%
\end{equation}
of $\operatorname*{Sp}(n)$ by the embedding%
\[
A+iB\longrightarrow%
\begin{pmatrix}
A & -B\\
B & A
\end{pmatrix}
.
\]

Recall \cite{Folland,Birk,HZ} that if $G\in\operatorname*{Sp}^{+}(n)$ then
there exists $U\in U(n)$ such that
\begin{equation}
G=U^{T}%
\begin{pmatrix}
\Lambda & 0\\
0 & \Lambda^{-1}%
\end{pmatrix}
U \label{gut}%
\end{equation}
where $\Lambda$ is the diagonal matrix whose diagonal elements are the $n$
eigenvalues $\geq1$ of $G$ (counting the multiplicities).

For further use we also recall the following classical result: let $N$ be a
(real) symmetric positive definite $2n\times2n$ matrix; then there exists
$S\in\operatorname*{Sp}(n)$ such that
\begin{equation}
S^{T}NS=%
\begin{pmatrix}
\Sigma & 0\\
0 & \Sigma
\end{pmatrix}
\label{william}%
\end{equation}
where $\Sigma$ is the diagonal matrix whose diagonal entries are the
symplectic eigenvalues of $N$, \textit{i.e.} the moduli of the eigenvalues
$\pm i\lambda$ ($\lambda>0$) of the product $JN$ (\textquotedblleft Williamson
diagonalization theorem\textquotedblright\ \cite{Will}; see
\cite{Folland,Birk,Birkbis,HZ} for proofs).

\subsection{A lemma, and its consequence\label{seclemma}}

Recall that an automorphism $M$ of $\mathbb{R}^{2n}$ is antisymplectic if
$\sigma(Mz,Mz^{\prime})=-\sigma(z,z^{\prime})$ for all $z,z^{\prime}%
\in\mathbb{R}^{2n}$; in matrix notation $M^{T}JM=-J$. Equivalently
$CM\in\operatorname*{Sp}(n)$ where $C=%
\begin{pmatrix}
I & 0\\
0 & -I
\end{pmatrix}
$.

\begin{lemma}
\label{Lemma}Let $M\in GL(2n,\mathbb{R})$ and assume that $M^{T}%
GM\in\operatorname*{Sp}(n)$ for every
\begin{equation}
G=%
\begin{pmatrix}
X & 0\\
0 & X^{-1}%
\end{pmatrix}
\in\operatorname*{Sp}\nolimits^{+}(n). \label{glambda}%
\end{equation}
Then $M$ is either symplectic, or anti-symplectic.
\end{lemma}

\begin{proof}
We first remark that, taking $G=I$ in the condition $M^{T}GM\in
\operatorname*{Sp}(n)$, we have $M^{T}M\in\operatorname*{Sp}(n)$. Next, we can
write $M=HP$ where $H=M(M^{T}M)^{-1/2}$ is orthogonal and $P=(M^{T}M)^{1/2}%
\in\operatorname*{Sp}^{+}(n)$ (polar decomposition theorem). It follows that
the condition $M^{T}GM\in\operatorname*{Sp}(n)$ is equivalent to
$P(H^{T}GH)P\in\operatorname*{Sp}(n)$; since $P$ is symplectic so is $P^{-1}$
and hence $H^{T} G H \in Sp(n)$ for all $G$ of the form (\ref{glambda}).

Let us now make the following particular choice for $G$: it is any diagonal
matrix%
\[
G=%
\begin{pmatrix}
\Lambda & 0\\
0 & \Lambda^{-1}%
\end{pmatrix}
\text{ \ , \ }\Lambda=\operatorname{diag}(\lambda_{1},...,\lambda_{n})
\]
with $\lambda_{j}>0$ for $1\leq j\leq n$. We thus have%
\[
H^{T}%
\begin{pmatrix}
\Lambda & 0\\
0 & \Lambda^{-1}%
\end{pmatrix}
H\in\operatorname*{Sp}\nolimits^{+}(n)
\]
for every $\Lambda$ of this form. Let $U\in U(n)$ be such that%
\[
H^{T}%
\begin{pmatrix}
\Lambda & 0\\
0 & \Lambda^{-1}%
\end{pmatrix}
H=U^{T}%
\begin{pmatrix}
\Lambda & 0\\
0 & \Lambda^{-1}%
\end{pmatrix}
U.
\]
($H^{T}GH\in\operatorname*{Sp}^{+}(n)$ and the eigenvalues of $H^{T}GH$ are
those of $G$ since $H$ is orthogonal) and set $R=HU^{T}$; the equality above
is equivalent to%
\begin{equation}%
\begin{pmatrix}
\Lambda & 0\\
0 & \Lambda^{-1}%
\end{pmatrix}
R=R%
\begin{pmatrix}
\Lambda & 0\\
0 & \Lambda^{-1}%
\end{pmatrix}
. \label{garag}%
\end{equation}
Writing $R=%
\begin{pmatrix}
A & B\\
C & D
\end{pmatrix}
$ we get the conditions%
\begin{align*}
\Lambda A  &  =A\Lambda\text{ , \ }\Lambda B=B\Lambda^{-1}\\
\Lambda^{-1}C  &  =C\Lambda\text{ \ , \ }\Lambda^{-1}D=D\Lambda^{-1}.
\end{align*}
for all $\Lambda$. It follows from these conditions that $A$ and $D$ must
themselves be diagonal $A=\operatorname{diag}(a_{1},...,a_{n})$,
$D=\operatorname{diag}(d_{1},...,d_{n})$. On the other hand, choosing
$\Lambda=\lambda I$, $\lambda\neq1$, we get $B=C=0$. Hence, taking into
account the fact that $R\in O(2n,\mathbb{R})$ we must have%
\begin{equation}
R=%
\begin{pmatrix}
A & 0\\
0 & D
\end{pmatrix}
\text{ \ , \ }A^{2}=D^{2}=I; \label{MatrixR}%
\end{equation}
Conversely, if $R$ is of the form (\ref{MatrixR}), then (\ref{garag}) holds
for any positive-definite diagonal $\Lambda$. We conclude that $M$ has to be
of the form $M=RUP$ where $R$ is of the form (\ref{MatrixR}). Since
$UP\in\operatorname*{Sp}(n)$, $M^{T}GM\in\operatorname*{Sp}(n)$ implies
$RGR\in\operatorname*{Sp}(n)$.

To proceed, for each pair $i,j$ with $1\leq i<j\leq n$ we choose the following
matrix $X$ in (\ref{glambda}):
\begin{equation}
X^{(ij)}=I+\tfrac{1}{2}E^{(ij)} \label{eq12}%
\end{equation}
where $E^{(ij)}$ is the symmetric matrix whose entries are all zero except the
ones on the $i$-th row and $j$-th column and on the $j$-th row and $i$-th
column which are equal to one. For instance if $n=4$, we have
\begin{equation}
X^{(13)}=\left(
\begin{array}
[c]{cccc}%
1 & 0 & \frac{1}{2} & 0\\
0 & 1 & 0 & 0\\
\frac{1}{2} & 0 & 1 & 0\\
0 & 0 & 0 & 1
\end{array}
\right)  . \label{eq13}%
\end{equation}
A simple calculation then reveals that
\begin{equation}
RGR=\left(
\begin{array}
[c]{cc}%
AX^{(ij)}A & 0\\
0 & D(X^{(ij)})^{-1}D
\end{array}
\right)  \label{eq10}%
\end{equation}
If we impose $RGR\in\operatorname*{Sp}(n)$, we obtain
\begin{equation}
AX^{(ij)}AD(X^{(ij)})^{-1}D=I\Longleftrightarrow X^{(ij)}AD=ADX^{(ij)}
\label{eq11}%
\end{equation}
In other words the matrix $AD$ commutes with every real positive-definite
$n\times n$ matrix $X^{(ij)}$ of the form (\ref{eq12}).

Let us write $AD=\operatorname{diag}(c_{1},\cdots,c_{n})$ with $c_{j}%
=a_{j}d_{j}$ for $1\leq$ $j\leq n$. Applying (\ref{eq11}) to (\ref{eq12}) for
$i<j$, we conclude that
\begin{equation}
c_{i}=c_{j}. \label{eq14}%
\end{equation}
This means that the entries of the matrix $AD$ are all equal, that is, either
$AD=I$ or $AD=-I$, or equivalently $A=D$ or $A=-D$. In the first case, $R$ is
symplectic and so is $M$. In the second case $R$ is anti-symplectic; but then
$M$ is also anti-symplectic.
\end{proof}

There is a very interesting link between Lemma \ref{Lemma} and symplectic
topology (in particular the notion of symplectic capacities of ellipsoids). In
fact it is proven in \cite{HZ} that the only linear mappings that preserve the
symplectic capacities \cite{Birkbis,HZ,Polter} of ellipsoids in the symplectic
space $(\mathbb{R}^{2n},\sigma)$ are either symplectic or antisymplectic.
Recall that the symplectic capacity of an ellipsoid $\Omega_{G}=\{z:Gz\cdot
z\leq1\}$ ($G$ a real symmetric positive-definite $2n\times2n$ matrix) can be
defined in terms of Gromov's width by%
\[
c(\Omega_{G})=\sup_{f\in\operatorname*{Symp}(n)}\{\pi r^{2}:f(B^{2n}%
(r))\subset\Omega_{G}\}
\]
where $B^{2n}(r)=\{z:|z|\leq r\}$ is the closed ball of radius $r$. The number
$c(\Omega_{G})$ is in practice calculated as follows: let $\lambda_{\max}$ be
the largest symplectic eigenvalue of $G$; then $c(\Omega_{G})=\pi
/\lambda_{\max}$. Now (\cite{HZ}, Theorem 5 p.61, and its Corollary, p.64)
assume that $f$ is a linear map $\mathbb{R}^{2n}\longrightarrow\mathbb{R}%
^{2n}$ such that $c(f(\Omega_{G}))=c(\Omega_{G})$ for all $G$. Then $f$ is
symplectic or antisymplectic. It turns out that our Lemma \ref{Lemma} yields a
sharper result: let us call symplectic ball the image of $B^{2n}(r)$ by an
element $S\in\operatorname*{Sp}(n)$. A symplectic ball $S(B^{2n}(r))$ is an
ellipsoid having symplectic capacity $c(S(B^{2n}(r)))=c(B^{2n}(r))=\pi r^{2}$. Then:

\begin{proposition}
\label{prop1}Let $k:\mathbb{R}^{2n}\longrightarrow\mathbb{R}^{2n}$ be a linear
automorphism taking any symplectic ball to a symplectic ball. Then $k$ is
either symplectic or antisymplectic.
\end{proposition}

\begin{proof}
The symplectic ball $S(B^{2n}(r))$ is defined by the inequality $Gz\cdot
z\leq1$ where $G=(1/r^{2})(S^{T})^{-1}S^{-1}\in\operatorname*{Sp}%
\nolimits^{+}(n)$. Let $K$ be the matrix of $k$ in the canonical basis; we
have%
\[
k(S(B^{2n}(r)))=\{z:(K^{-1})^{T}GK^{-1}z\cdot z\leq1\}
\]
hence $k(S(B^{2n}(r)))$ is a symplectic ball if and only if $(K^{-1}%
)^{T}GK^{-1}\in\operatorname*{Sp}\nolimits^{+}(n)$. The proof now follows from
Lemma \ref{Lemma} with $M=K^{-1}$.
\end{proof}

\section{The main result}

Let us now prove our main result:

\begin{theorem}
\label{Theo1}Let $M\in GL(2n,\mathbb{R})$.

(i) Assume that $M$ is antisymplectic: $S=CM\in\operatorname*{Sp}(n)$ where
$C=%
\begin{pmatrix}
I & 0\\
0 & -I
\end{pmatrix}
$; then for every $\psi\in\mathcal{S}^{\prime}(\mathbb{R}^{n})$
\begin{equation}
W\psi(Mz)=W(\widehat{S}^{-1}\overline{\psi})(z) \label{mc}%
\end{equation}
where $\widehat{S}$ is any of the two elements of $\operatorname*{Mp}(n)$
covering $S$.

(ii) Conversely, assume that for any $\psi\in\mathcal{S}(\mathbb{R}^{n})$
there exists $\psi^{\prime}\in\mathcal{S}^{\prime}(\mathbb{R}^{n})$ such that%
\begin{equation}
W\psi(Mz)=W\psi^{\prime}(z). \label{mcond}%
\end{equation}
Then $M$ is either symplectic or antisymplectic.
\end{theorem}

\begin{proof}
(i) It is sufficient to assume that $\psi\in\mathcal{S}(\mathbb{R}^{n})$. We
have%
\begin{align*}
W\psi(Cz)  &  =\left(  \tfrac{1}{2\pi\hbar}\right)  ^{n}\int_{\mathbb{R}^{n}%
}e^{\frac{i}{\hbar}p\cdot y}\psi(x+\tfrac{1}{2}y)\overline{\psi}(x-\tfrac
{1}{2}y)dy\\
&  =\left(  \tfrac{1}{2\pi\hbar}\right)  ^{n}\int_{\mathbb{R}^{n}}e^{-\frac
{i}{\hbar}p\cdot y}\psi(x-\tfrac{1}{2}y)\overline{\psi}(x+\tfrac{1}{2}y)dy\\
&  =W\overline{\psi}(z).
\end{align*}
It follows that
\[
W\psi(Mz)=W\psi(CSz)=W\overline{\psi}(Sz)
\]
hence formula (\ref{mc}). (ii) Choosing for $\psi$ a Gaussian of the form
\begin{equation}
\psi_{X}(x)=\left(  \tfrac{1}{\pi\hbar}\right)  ^{n/4}(\det X)^{1/4}%
e^{-\tfrac{1}{2\hbar}Xx\cdot x} \label{gausshud}%
\end{equation}
($X$ is real symmetric and positive definite) we have
\begin{equation}
W\psi_{X}(z)=\left(  \tfrac{1}{\pi\hbar}\right)  ^{n}e^{-\tfrac{1}{\hbar
}Gz\cdot z} \label{phagauss}%
\end{equation}
where%
\begin{equation}
G=%
\begin{pmatrix}
X & 0\\
0 & X^{-1}%
\end{pmatrix}
\label{gsym}%
\end{equation}
is positive definite and belongs to $\operatorname*{Sp}(n)$. Condition
(\ref{mcond}) implies that we must have%
\[
W\psi^{\prime}(z)=\left(  \tfrac{1}{\pi\hbar}\right)  ^{n}e^{-\tfrac{1}{\hbar
}M^{T}GMz\cdot z}.
\]
The Wigner transform of a function being a Gaussian if and only if the
function itself is a Gaussian (see \cite{Birk,Birkbis}). This can be seen in
the following way: the matrix $M^{T}GM$ being symmetric and positive definite
we can use a Williamson diagonalization (\ref{william}): there exists
$S\in\operatorname*{Sp}(n)$ such that%
\begin{equation}
S^{T}(M^{T}GM)S=\Delta=%
\begin{pmatrix}
\Sigma & 0\\
0 & \Sigma
\end{pmatrix}
\text{\ } \label{sms}%
\end{equation}
and hence%
\[
W\psi^{\prime}(Sz)=\left(  \tfrac{1}{\pi\hbar}\right)  ^{n}e^{-\tfrac{1}%
{\hbar}\Delta z\cdot z}.
\]
In view of the symplectic covariance of the Wigner transform, we have
\[
W\psi^{\prime}(Sz)=W\psi^{\prime\prime}(z)\text{ \ , \ }\psi^{\prime\prime
}=\widehat{S}^{-1}\psi^{\prime}%
\]
where $\widehat{S}\in\operatorname*{Mp}(n)$ is one of the two elements of the
metaplectic group covering $S$. We now show that the equality%
\begin{equation}
W\psi^{\prime\prime}(z)=\left(  \tfrac{1}{\pi\hbar}\right)  ^{n}e^{-\tfrac
{1}{\hbar}\Delta z\cdot z} \label{delta}%
\end{equation}
implies that $\psi^{\prime\prime}$ must be a Gaussian of the form
(\ref{gausshud}) and hence $W\psi^{\prime\prime}$ must be of the type
(\ref{phagauss},\ref{gsym}). That $\psi^{\prime\prime}$ must be a Gaussian
follows from $W\psi^{\prime\prime}\geq0$ and Hudson's theorem (see
\textit{e.g.} \cite{Folland}). If $\psi^{\prime\prime}$ were of the more
general type
\begin{equation}
\psi_{X,Y}(x)=\left(  \tfrac{1}{\pi\hbar}\right)  ^{n/4}(\det X)^{1/4}%
e^{-\tfrac{1}{2\hbar}(X+iY)x\cdot x} \label{psixy}%
\end{equation}
($X,Y$ are real and symmetric and $X$ is positive definite) the matrix $G$ in
(\ref{phagauss}) would be
\begin{equation}
G=%
\begin{pmatrix}
X+YX^{-1}Y & YX^{-1}\\
X^{-1}Y & X^{-1}%
\end{pmatrix}
\end{equation}
which is only compatible with (\ref{delta}) if $Y=0$. In addition, due to the
parity of $W\psi^{\prime\prime}$, $\psi^{\prime\prime}$ must be even hence
Gaussians more general than $\psi_{X,Y}$ are excluded. It follows from these
considerations that we have%
\[
\Delta=%
\begin{pmatrix}
\Sigma & 0\\
0 & \Sigma
\end{pmatrix}
=%
\begin{pmatrix}
X & 0\\
0 & X^{-1}%
\end{pmatrix}
\]
so that $\Sigma=\Sigma^{-1}$. Since $\Sigma>0$ this implies that we must have
$\Sigma=I$, and hence, using formula (\ref{sms}), $S^{T}(M^{T}GM)S=I$. It
follows that we must have $M^{T}GM\in\operatorname*{Sp}(n)$ for every $G=%
\begin{pmatrix}
X & 0\\
0 & X^{-1}%
\end{pmatrix}
\in\operatorname*{Sp}^{+}(n)$. In view of Lemma \ref{Lemma} the matrix $M$
must then be either symplectic or antisymplectic.
\end{proof}

\begin{remark}
An alternative way of proving that (\ref{delta}) implies that $\Sigma=I$ is to
use the formulation of Hardy's uncertainty principle \cite{ha33} for Wigner
transforms introduced in \cite{hardy} (see \cite{Birkbis}, Theorem 105, for a
detailed study).
\end{remark}

\section{Application to Weyl Operators}

\subsection{The Weyl correspondence}

Let $a\in\mathcal{S}(\mathbb{R}^{2n})$; the operator $\widehat{A}$ defined for
all $\psi\in\mathcal{S}(\mathbb{R}^{n})$ by
\begin{equation}
\widehat{A}\psi(x)=\left(  \tfrac{1}{2\pi\hbar}\right)  ^{n}\iint
\nolimits_{\mathbb{R}^{2n}}e^{\frac{i}{\hbar}p\cdot(x-y)}a(\tfrac{1}%
{2}(x+y),p)\psi(y)dydp \label{tallyweyl}%
\end{equation}
is called the Weyl operator with symbol $a$. For more general symbols
$a\in\mathcal{S}^{\prime}(\mathbb{R}^{2n})$ one can define $\widehat{A} \psi$
in a variety of ways \cite{Folland,Birk}; we will see one below. The Weyl
correspondence $a\overset{\text{Weyl}}{\longleftrightarrow} \widehat{A}$ is
linear and one-to-one: If $a\overset{\text{Weyl}}{\longleftrightarrow}%
\widehat{A}$ and $a^{\prime}\overset{\text{Weyl}}{\longleftrightarrow}
\widehat{A}$ then $a=a^{\prime}$, and we have $1\overset{\text{Weyl}%
}{\longleftrightarrow} I$ where $I$ is the identity operator on $\mathcal{S}%
^{\prime}(\mathbb{R}^{n})$.

There is a fundamental relation between Weyl operators and the cross-Wigner
transform, which is a straightforward generalization of the Wigner transform
\cite{Wigner}: it is defined, for $\psi,\phi\in\mathcal{S}(\mathbb{R}^{n})$ by%
\[
W(\psi,\phi)(z)=\left(  \tfrac{1}{2\pi\hbar}\right)  ^{n}\int_{\mathbb{R}^{n}%
}e^{-\frac{i}{\hbar}p\cdot y}\psi(x+\tfrac{1}{2}y)\overline{\phi}(x-\tfrac
{1}{2}y)dy
\]
(in particular $W(\psi,\psi)=W\psi$). In fact, if $\widehat{A}\overset
{\mathrm{Weyl}}{\longleftrightarrow}a$ then%
\begin{equation}
\langle\widehat{A}\psi,\overline{\phi}\rangle=\langle\langle a,W(\psi
,\phi)\rangle\rangle\label{aw}%
\end{equation}
where $\langle\cdot,\cdot\rangle$ is the distributional bracket on
$\mathbb{R}^{n}$ and $\langle\langle\cdot,\cdot\rangle\rangle$ that on
$\mathbb{R}^{2n}$; the latter pairs distributions $\Psi\in\mathcal{S}^{\prime
}(\mathbb{R}^{2n})$ and Schwartz functions $\Phi\in\mathcal{S}(\mathbb{R}%
^{2n})$; when $\Psi\in L^{2}(\mathbb{R}^{2n})$ we thus have%
\[
\langle\langle\Psi,\Phi\rangle\rangle=%
{\displaystyle\int\nolimits_{\mathbb{R}^{2n}}}
\Psi(z)\Phi(z)dz.
\]
This relation can actually be taken as a concise definition of an arbitrary
Weyl operator $\widehat{A}:\mathcal{S}(\mathbb{R}^{n})\longrightarrow
\mathcal{S}^{\prime}(\mathbb{R}^{n})$; for $\psi,\phi\in\mathcal{S}%
(\mathbb{R}^{n})$ we have $W(\psi,\phi)\in\mathcal{S}(\mathbb{R}^{2n})$ the
right-hand side is defined for arbitrary $a\in\mathcal{S}^{\prime}%
(\mathbb{R}^{2n})$ and this defines unambiguously $\widehat{A}\psi$ since
$\phi$ is arbitrary. The symplectic covariance property for Weyl operators%
\begin{equation}
\widehat{S}^{-1}\widehat{A}\widehat{S}\overset{\mathrm{Weyl}}%
{\longleftrightarrow}a\circ S \label{sas}%
\end{equation}
actually easily follows: since
\begin{equation}
W(\widehat{S}\psi,\widehat{S}\phi)(z)=W(\psi,\phi)(S^{-1}z) \label{wpf}%
\end{equation}
we have%
\begin{align*}
\langle\langle a\circ S,W(\psi,\phi)\rangle\rangle &  =\langle\langle
a,W(\psi,\phi)\circ S^{-1}\rangle\rangle\\
&  =\langle\langle a,W(\widehat{S}\psi,\widehat{S}\phi)\rangle\rangle\\
&  =\langle\widehat{A}\widehat{S}\psi,\overline{\widehat{S}\phi}\rangle\\
&  =\langle\widehat{S}^{-1}\widehat{A}\widehat{S}\psi,\overline{\phi}\rangle
\end{align*}
which proves the covariance relation (\ref{sas}).

\subsection{Maximal covariance of Weyl operators}

Theorem \ref{Theo1} implies the following maximal covariance result for Weyl operators:

\begin{corollary}
\label{coro1}Let $M\in GL(2n,\mathbb{R})$. Assume that there exists a unitary
operator $\widehat{M}:L^{2}(\mathbb{R}^{n})\longrightarrow L^{2}%
(\mathbb{R}^{n})$ such that $\widehat{M}\widehat{A}\widehat{M}^{-1}%
\overset{\mathrm{Weyl}}{\longleftrightarrow}a\circ M^{-1}$ for all
$\widehat{A}\overset{\mathrm{Weyl}}{\longleftrightarrow}a\in\mathcal{S}%
(\mathbb{R}^{2n})$. Then $M$ is symplectic or antisymplectic.
\end{corollary}

\begin{proof}
Suppose that $\widehat{M}\widehat{A}\widehat{M}^{-1}\overset{\mathrm{Weyl}%
}{\longleftrightarrow}a\circ M^{-1}$; then, by (\ref{aw}),%
\begin{align*}
(\widehat{M}\widehat{A}\widehat{M}^{-1}\psi|\phi)_{L^{2}}  &  =\langle\langle
a\circ M^{-1},W(\psi,\phi)\rangle\rangle\\
&  =\langle\langle a,W(\psi,\phi)\circ M\rangle\rangle.
\end{align*}
On the other hand, using the unitarity of $\widehat{M}$ and (\ref{aw}),%
\begin{align*}
(\widehat{M}\widehat{A}\widehat{M}^{-1}\psi|\phi)_{L^{2}}  &  =(\widehat
{A}\widehat{M}^{-1}\psi|\widehat{M}^{-1}\phi)_{L^{2}}\\
&  =\langle\langle a,W(\widehat{M}^{-1}\psi,\widehat{M}^{-1}\phi
)\rangle\rangle.
\end{align*}
It follows that we must have
\[
\langle\langle a,W(\psi,\phi)\circ M\rangle\rangle=\langle\langle
a,W(\widehat{M}^{-1}\psi,\widehat{M}^{-1}\phi)\rangle\rangle
\]
for all $\psi,\phi\in\mathcal{S}(\mathbb{R}^{n})$ and hence, in particular,
taking $\psi=\phi$:%
\[
\langle\langle a,W\psi\circ M\rangle\rangle=\langle\langle a,W(\widehat
{M}^{-1}\psi)\rangle\rangle
\]
for all $\psi\in\mathcal{S}(\mathbb{R}^{n})$. Since $a$ is arbitrary this
implies that we must have $W\psi\circ M=W(\widehat{M}^{-1}\psi)$. In view of
Theorem \ref{Theo1} the automorphism $M$ must be symplectic or antisymplectic.
\end{proof}

\section{Discussion and Concluding Remarks}

The results above, together with those in \cite{JPDOA}, where it was proved
that one cannot expect a covariance formula for non-linear symplectomorphisms,
show that the symplectic group indeed is a maximal linear covariance group for
both Wigner transforms and general Weyl pseudo-differential operators. As
briefly mentioned in the Introduction, one can prove \cite{JPDOA,transam}
partial symplectic covariance results for other classes of pseudo-differential
operators (Shubin, or Born--Jordan operators). Corollary \ref{coro1} proves
that one cannot expect to extend these results to more general linear
non-symplectic automorphisms.

The link between Lemma \ref{Lemma}, its consequence, Proposition \ref{prop1},
and the notion of symplectic capacity of ellipsoids in the symplectic space
$(\mathbb{R}^{2n},\sigma)$ is not after all so surprising: as has been shown
in \cite{Birkbis,hardy,physreps} there is a deep and certainly essential
interplay between Weyl calculus, the theory of Wigner transforms, the
uncertainty principle, and Gromov's non-squeezing theorem \cite{Gromov}. For
instance, the methods used in this paper can be used to show that the
uncertainty principle in its strong Robertson--Schr\"{o}dinger form
\cite{physreps} is only invariant under symplectic or antisymplectic transforms.

\begin{acknowledgement}
Nuno Costa Dias and Jo\~ao Nuno Prata have been supported by the research
grant PTDC/MAT/099880/2008 of the Portuguese Science Foundation. Maurice de
Gosson has been supported by a research grant from the Austrian Research
Agency FWF (Projektnummer P23902-N13).
\end{acknowledgement}

\begin{acknowledgement}
The authors wish to thank the referee for useful remarks and suggestions, and
for having pointed out some misprints.
\end{acknowledgement}

\pagebreak
\end{document}